\documentclass[12pt,a4paper]{amsart}

\usepackage{amsmath,amsfonts,amssymb}

\usepackage[hmargin=2cm,vmargin=2cm]{geometry}

\usepackage{hyperref}
\usepackage{nameref,zref-xr}                    

\newcommand{\mbP}{\mathbb P}
\newcommand{\mbZ}{\mathbb Z}
\newcommand{\mbC}{\mathbb C}

\newcommand{\oM}{\overline{\mathcal M}}

\newcommand{\og}{\overline g}

\def\cM{{\mathcal{M}}}
\def\oM{{\overline{\mathcal{M}}}}

\renewcommand{\Im}{\mathrm{Im}}

\def\d{{\partial}}

\newcommand{\eps}{\varepsilon}

\newcommand{\DR}{\mathrm{DR}}

\newcommand{\even}{\mathrm{even}}
\newcommand{\ct}{\mathrm{ct}}

\newcommand{\gl}{\mathrm{gl}}

\newcommand{\oa}{\overline{a}}
\newcommand{\ob}{\overline{b}}

\newcommand{\mcP}{\mathcal{P}}
\newcommand{\tlambda}{{\widetilde{\lambda}}}
\newcommand{\hlambda}{{\widehat{\lambda}}}

\newcommand{\Sym}{\mathrm{Sym}}
\newcommand{\td}{{\widetilde{d}}}
\newcommand{\tmu}{{\widetilde{\mu}}}
\newcommand{\hmu}{{\widehat{\mu}}}
\newcommand{\of}{{\overline f}}

\DeclareMathOperator{\tdeg}{\widetilde{\mathrm{deg}}}

\newtheorem{theorem}{Theorem}[section]
\newtheorem{proposition}[theorem]{Proposition}
\newtheorem{lemma}[theorem]{Lemma}

\theoremstyle{definition}

\newtheorem{definition}[theorem]{Definition}

\numberwithin{equation}{section}

\begin{document}

\title[Quadratic DR integrals and the noncommutative KdV hierarchy]{Quadratic double ramification integrals and the noncommutative KdV hierarchy}

\author{Alexandr Buryak}
\address{A. Buryak:\newline 
Faculty of Mathematics, National Research University Higher School of Economics, \newline
6 Usacheva str., Moscow, 119048, Russian Federation; \newline 
Center for Advanced Studies, Skolkovo Institute of Science and Technology, \newline
1 Nobel str., Moscow, 143026, Russian Federation; \newline
Novosibirsk State University, ul. Pirogova, 1, Novosibirsk, 630090, Russian Federation}
\email{aburyak@hse.ru}

\author{Paolo Rossi}
\address{P.~Rossi:\newline Dipartimento di Matematica ``Tullio Levi-Civita'', Universit\`a degli Studi di Padova,\newline Via Trieste 63, 35121 Padova, Italy}
\email{paolo.rossi@math.unipd.it}

\begin{abstract}
In this paper we compute the intersection number of two double ramification cycles (with different ramification profiles) and the top Chern class of the Hodge bundle on the moduli space of stable curves of any genus. These quadratic double ramification integrals are the main ingredient for the computation of the double ramification hierarchy associated to the infinite dimensional partial cohomological field theory given by $\exp(\mu^2 \Theta)$ where $\mu$ is a parameter and~$\Theta$ is Hain's theta class, appearing in Hain's formula for the double ramification cycle on the moduli space of curves of compact type. This infinite rank double ramification hierarchy can be seen as a rank $1$ integrable system in two space and one time dimensions. We prove that it coincides with a natural analogue of the KdV hierarchy on a noncommutative Moyal torus.
\end{abstract}

\date{\today}

\maketitle

\tableofcontents

\section{Introduction}

The main idea of this paper comes from the observation that the double ramification (DR) cycle, i.e. the class in the cohomology of the moduli space of stable curves representing the most natural compactification of the locus of smooth curves whose marked points support a principal divisor \cite{Hai13}, can be seen as a partial cohomological field theory (CohFT) \cite{LRZ15,KM94} with an infinite dimensional phase space.\\

In \cite{Bur15,BR16a,BDGR18} it was shown how to associate to any partial CohFT an integrable hierarchy of Hamiltonian systems of evolutionary PDEs in one space and one time dimensions. The number of dependent variables in these system of PDEs equals the dimension of the phase space of the partial CohFT. This integrable system is called the DR hierarchy and its properties and generalizations (including a quantization which exists in the case of actual CohFTs) where studied in \cite{BR16b,BDGR16,BGR17,BR18}.\\

This paper wants to answer the question: ``what is the DR hierarchy associated to the infinite rank partial CohFT given by the DR cycle?''. As the general construction of the DR hierarchy already involves intersection numbers of a partial CohFT with the top Hodge class and a DR cycle, choosing as CohFT a second DR cycle leads us directly to having to compute intersection numbers of two different DR cycles and the top Hodge class. This is, of course, a question of its own geometric interest and, as we show in Chapter~\ref{section:quadratic} of this paper, it has a very explicit answer.\\

In fact there is a natural deformation of the DR cycle which gives a one-parameter family of partial CohFTs. It comes from Hain's formula \cite{Hai13}, expressing the DR cycle (restricted to the moduli space of stable curves of compact type) as the $g$-th power of (the pullback to the moduli space of stable curves of compact type) of the class of the theta divisor on the universal Jacobian. If we consider instead the exponential of such theta class, putting into play all of its powers, we get a more general but still very explicit infinite rank partial CohFT.\\

After showing in Chapter \ref{section:DR} how to trade the resulting infinite rank DR hierarchy in one space and one time dimensions as a rank $1$ hierarchy in two space and one time dimensions, we set out to compute it explicitly. The main result of this paper is in Chapter \ref{section:ncKdV}, where we show that the DR hierarchy of our infinite rank partial CohFT coincides with the noncommutative KdV hierarchy, the natural generalization of the ordinary KdV hierarchy on a circle to a torus with a noncommutative Moyal structure.\\

One of the applications is that, through our mixed intersection theory and integrable systems techniques, we are able to compute the intersection numbers of the top Hodge class, one DR cycle, any power of the theta class and any power of the psi class at one marked point. Note also that noncommutative integrable systems have never appeared before in the study of the intersection theory on the moduli spaces of curves. Therefore, we expect that our result will provide a new tool for understanding the structure of the cohomology ring of the moduli spaces and, in particular, the structure of the DR cycle, which became the object of an intensive research in recent years (see e.g.~\cite{JPPZ17,HPS19,Pix18}).\\

\noindent{\bf Acknowledgements}. We would like to thank Johannes Schmitt for providing us with an early version of the SageMath package \texttt{admcycles}, presented in \cite{DSZ20}. The program was used for preliminary computational experiments in the early phases of this paper. The work of A.~B. was supported by Mathematical Center in Akademgorodok under agreement No. 075-2019-1675 with the Ministry of Science and Higher Education of the Russian Federation.


\section{Quadratic double ramification integrals}\label{section:quadratic}

For a pair of nonnegative integers $(g,n)$ in the stable range, i.e. satisfying $2g+2-n>0$, let $\oM_{g,n}$ be the moduli space of stable curves with genus $g$ and $n$ marked points labeled by the set $[n]:=\{1,\ldots,n\}$. For integers $a_1,\ldots,a_n$, such that $\sum a_i=0$, the double ramification cycle $\DR_g(a_1,\ldots,a_n) \in H^*(\oM_{g,n},\mbC)$ is the Poincar\'e dual to the pushforward to $\oM_{g,n}$ of the virtual fundamental class of the moduli space of rubber stable maps from curves of genus $g$ with $n$ marked points to $\mbP^1$ relative to $0$ and $\infty$ with ramification profile given by $(a_1,\dots,a_n)$. Here ``rubber'' means that we consider maps up to the $\mbC^*$-action in the target $\mbP^1$ and a~positive (negative) coefficient~$a_i$ indicates a pole (zero) at the $i$-th marked point of order $a_i$ ($-a_i$), while $a_i=0$ just indicates that the $i$-th marked point is unconstrained. For future convenience we will also define the class $\DR_g(a_1,\ldots,a_n)$ to vanish in case $\sum a_i\neq0$.\\

Let us introduce the class $\Theta(a_1,\ldots,a_n)\in H^2(\oM_{g,n},\mbC)$ defined by
\begin{gather}\label{eq:Theta-class}
\Theta(a_1,\ldots,a_n):=\sum_{j=1}^n\frac{a_j^2\psi_j}{2}-\frac{1}{4}\sum_{h=0}^g\sum_{J\subset [n]}a_J^2\delta_h^J,
\end{gather}
if $\sum a_i=0$ and zero otherwise, where $\psi_i$, $1,\leq i\leq n$, is the first Chern class of the $i$-th tautological line bundle, and, for $J\subset [n]$ and $0\leq h\leq g$ in the stable range $2h-1+|J|>0$ and $2(g-h)-1+(n-|J|)>0$, $a_J:=\sum_{j\in J}a_j$ and $\delta^J_h \in H^2(\oM_{g,n},\mbC)$ is the class of the irreducible boundary divisor of $\oM_{g,n}$ formed by stable curves with a separating node at which two stable components meet, one of genus $h$ and marked points labeled by $|J|$ and the other of genus~$g-h$ and marked points labeled by the complement $J^c$ (naturally, $\delta^J_h=0$ if at least one of the stability conditions $2h-1+|J|>0$ and $2(g-h)-1+(n-|J|)>0$ is not satisfied).\\

By a result of Hain \cite{Hai13}, we know that
\begin{gather}\label{eq:Hain}
\left.\DR_g(a_1,\ldots,a_n)\right|_{\cM_{g,n}^\ct}=\frac{1}{g!}\Theta(a_1,\ldots,a_n)^g|_{\cM_{g,n}^\ct}.
\end{gather}
where $\cM_{g,n}^\ct\in \oM_{g,n}$ is the locus of stable curves with no non-separating nodes. Moreover, always by \cite{Hai13}, the class $\Theta(a_1,\ldots,a_n)|_{\cM_{g,n}^\ct}$ represents the pullback to $\cM_{g,n}^\ct$ of the theta divisor on the universal Jacobian over $\cM_{g,n}^\ct$, which implies the following relation in~$H^*(\cM_{g,n}^\ct,\mbC)$,
\begin{gather}\label{eq:Theta-relation}
\left.\Theta(a_1,\ldots,a_n)^{g+1}\right|_{\cM^\ct_{g,n}}=0.
\end{gather}

Let $\lambda_i\in H^{2i}(\oM_{g,n},\mbC)$, $1\leq i \leq g$, be the $i$-th Chern class of the Hodge bundle on $\oM_{g,n}$. We have \cite{FP00}
\begin{gather}\label{eq:lambdag-relation}
\lambda_g|_{\oM_{g,n}\setminus \cM_{g,n}^\ct} = 0.
\end{gather}
The classes $\lambda_i$, $\DR_g(a_1,\ldots,a_n)$ and $\Theta(a_1,\ldots,a_n)$ are algebraic, i.e. they belong in fact to the Chow ring $A^*(\oM_{g,n})$ . By the localization exact sequence, see e.g. \cite[Section~1.8]{Ful98}, for all~$k$, 
$$A_k(\oM_{g,n} \setminus \cM_{g,n}^\ct) \xrightarrow{i_*} A_k(\oM_{g,n}) \xrightarrow{j^*} A_k (\cM_{g,n}^\ct) \to 0,$$
where $i$ and $j$ are the inclusion maps of $\oM_{g,n} \setminus \cM_{g,n}^\ct$ and $\cM_{g,n}^\ct$ into $\oM_{g,n}$, and by equation (\ref{eq:lambdag-relation}), we deduce that, if $\alpha \in A_k(\oM_{g,n})$ is such that $\alpha|_{\cM_{g,n}^\ct}:= j^*\alpha=0$, then $\lambda_g \cdot \alpha = 0 \in A^*(\oM_{g,n})$. This allows to deduce from identities (\ref{eq:Hain}) and (\ref{eq:Theta-relation}) the following relations in $H^*(\oM_{g,n},\mbC)$:
\begin{align}
&\lambda_g \DR_g(a_1,\ldots,a_n)=\frac{1}{g!}\lambda_g \Theta(a_1,\ldots,a_n)^g, \label{eq:lambdag-Hain} \\
&\lambda_g \Theta(a_1,\ldots,a_n)^{g+1}=0. \label{eq:lambdag-Theta-relation}
\end{align}\\

For two pairs of nonnegative integers $(g_1,n_1+1)$ and $(g_2,n_2+1)$ in the stable range, let $\gl\colon\oM_{g_1,n_1+1}\times \oM_{g_2,n_2+1} \to \oM_{g,n}$ be the map that glues two stable curves of genus $g_1$ and $g_2$ with marked points labeled by the sets $I=\{i_1,\ldots,i_{n_1},{n+1}\}$ and $J=\{j_1,\ldots,j_{n_2},n+2\}$, respectively, with $I\sqcup J=[n+2]$, at their last marked points to form stable curve with a separating node with genus $g=g_1+g_2$, and marked points labeled by $[n]$. It is easy to see from the definitions that we have
\begin{gather}\label{eq:gluing-Theta}
\gl^*\Theta(a_1,\ldots,a_n) = \Theta\left(a_{i_1},\ldots,a_{i_{n_1}},-k\right) + \Theta\left(a_{j_1},\ldots,a_{j_{n_2}},k\right),
\end{gather}
where $k=\sum_{k=1}^{n_1}a_{i_k}=-\sum_{k=1}^{n_2}a_{j_k}$ and the classes on the right-hand side are pulled back from each of the two factors in the product $\oM_{g_1,n_1+1}\times \oM_{g_2,n_2+1} $. By~\cite{BSSZ15}, we also have
\begin{gather*}
\gl^*\DR_g(a_1,\ldots,a_n) = \DR_{g_1}\left(a_{i_1},\ldots,a_{i_{n_1}},-k\right) \DR_{g_2}\left(a_{j_1},\ldots,a_{j_{n_2}},k\right).
\end{gather*}
In the following we will denote by $\DR_{g_1}\left(a_{i_1},\ldots,a_{i_{n_1}},-k\right)\boxtimes \DR_{g_2}\left(a_{j_1},\ldots,a_{j_{n_2}},k\right)$ the pushforward $\gl_*\gl^*\DR_g(a_1,\ldots,a_n) \in H^*(\oM_{g,n},\mbC)$.\\

We have the following result on the intersection number of two double ramification cycles and the top Chern class of the Hodge bundle.
\begin{theorem}\label{eq:quadratic DR integral} Let $a_1,a_2,b_1,b_2 \in \mbZ$ and $g\in\mbZ_{\geq 0}$, then
\begin{equation}
\int_{\oM_{g,3}}\lambda_g\DR_g(a_1,a_2,-a_1-a_2)\DR_g(b_1,b_2,-b_1-b_2)=\frac{(a_1b_2-a_2b_1)^{2g}}{2^{3g}g!(2g+1)!!}.
\end{equation}
\end{theorem}
\begin{proof}
By equation (\ref{eq:lambdag-Hain}), Theorem \ref{eq:quadratic DR integral} is equivalent to the formula
\begin{gather}\label{eq:DRTheta integral}
\int_{\oM_{g,3}}\lambda_g\Theta(a_1,a_2,-a_1-a_2)^g\DR_g(b_1,b_2,-b_1-b_2)=\frac{(a_1b_2-a_2b_1)^{2g}}{2^{3g}(2g+1)!!},\quad g\ge 0.
\end{gather}
Denote the left-hand side of equation~\eqref{eq:DRTheta integral} by $f_g(\oa,\ob)$, where $\oa=(a_1,a_2,a_3)$, $\ob=(b_1,b_2,b_3)$ and $a_3=-a_1-a_2$, $b_3=-b_1-b_2$. Clearly, $f_0(\oa,\ob)=1$, so suppose that $g\ge 1$. In order to compute the integral, let us express one of the classes $\Theta(\oa)$ using formula~\eqref{eq:Theta-class}. Using relation~\eqref{eq:Theta-relation} and the formula for the intersection of a psi class with a double ramification cycle~\cite{BSSZ15}, for any $1\le i\le 3$ we compute
\begin{align*}
\int_{\oM_{g,3}}\lambda_g\psi_i\Theta(\oa)^{g-1}\DR_g(\ob)=&\frac{2g-1}{2g+1}\int_{\oM_{g,3}}\lambda_g\Theta(\oa)^{g-1}\DR_1(b_i,-b_i)\boxtimes\DR_{g-1}(b_j,b_k,b_i)+\\
&-\frac{2b_j}{(2g+1)b_i}\int_{\oM_{g,3}}\lambda_g\Theta(\oa)^{g-1}\DR_{g-1}(b_i,b_k,b_j)\boxtimes\DR_1(b_j,-b_j)+\\
&-\frac{2b_k}{(2g+1)b_i}\int_{\oM_{g,3}}\lambda_g\Theta(\oa)^{g-1}\DR_{g-1}(b_i,b_j,b_k)\boxtimes\DR_1(b_k,-b_k)=\\
=&\left[\frac{2g-1}{2g+1}\frac{b_i^2}{24}+\frac{2b_j}{(2g+1)(b_j+b_k)}\frac{b_j^2}{24}+\frac{2b_k}{(2g+1)(b_j+b_k)}\frac{b_k^2}{24}\right]f_{g-1}(\oa,\ob)=\\
=&\frac{(2g+1)b_i^2-6b_jb_k}{24(2g+1)}f_{g-1}(\oa,\ob),
\end{align*}
where $\{i,j,k\}=\{1,2,3\}$. As a result,
\begin{gather*}
\int_{\oM_{g,3}}\lambda_g\left(\sum_{i=1}^3a_i^2\psi_i\right)\Theta(\oa)^{g-1}\DR_g(\ob)=f_{g-1}(\oa,\ob)\sum_{\substack{i;\,j<k\\\{i,j,k\}=\{1,2,3\}}}a_i^2\frac{(2g+1)b_i^2-6b_jb_k}{24(2g+1)}.
\end{gather*}
Next we compute
\begin{align*}
&\int_{\oM_{g,3}}\lambda_g\left(\frac{1}{2}\sum_{h=0}^g\sum_{J\subset [3]}a_J^2\delta_h^J\right)\Theta(\oa)^{g-1}\DR_g(\ob)=\\
&\hspace{2cm}=\sum_{\substack{i;\,j<k\\\{i,j,k\}=\{1,2,3\}}}a_i^2\int_{\oM_{g,3}}\lambda_g\Theta(\oa)^{g-1}\DR_1(b_i,-b_i)\boxtimes\DR_{g-1}(b_j,b_k,b_i)=\\
&\hspace{2cm}=f_{g-1}(\oa,\ob)\sum_{i=1}^3\frac{a_i^2b_i^2}{24}.
\end{align*}
Summarizing the above computations we get
\begin{align*}
f_g(\oa,\ob)=&\frac{1}{2}\int_{\oM_{g,3}}\lambda_g\left(\sum_{i=1}^3a_i^2\psi_i-\frac{1}{2}\sum_{h=0}^g\sum_{J\subset [3]}a_J^2\delta_h^J\right)\Theta(\oa)^{g-1}\DR_g(\ob)=\\
=&\frac{1}{2}f_{g-1}(\oa,\ob)\left(\sum_{\substack{i;\,j<k\\\{i,j,k\}=\{1,2,3\}}}a_i^2\frac{(2g+1)b_i^2-6b_jb_k}{24(2g+1)}-\sum_{i=1}^3\frac{a_i^2b_i^2}{24}\right)=\\
=&-\frac{f_{g-1}(\oa,\ob)}{8(2g+1)}\sum_{\substack{i;\,j<k\\\{i,j,k\}=\{1,2,3\}}}a_i^2 b_jb_k=\\
=&\frac{(a_1b_2-a_2b_1)^2}{8(2g+1)}f_{g-1}(\oa,\ob),
\end{align*}
that proves formula~\eqref{eq:DRTheta integral} and completes the proof of the theorem.
\end{proof}


\section{An infinite rank partial CohFT and its DR hierarchy}\label{section:DR}

Recall the following definition, which is a generalization first considered in \cite{LRZ15} of the notion of cohomological field theory (CohFT) from \cite{KM94}.

\begin{definition}
For a pair of nonnegative integers $(g,n)$ in the stable range $2g-2+n>0$, a partial CohFT is a system of linear maps $c_{g,n}\colon V^{\otimes n} \to H^\even(\oM_{g,n},\mbC)$, where $V$ is an arbitrary finite dimensional $\mbC$-vector space, called the phase space, together with a special element $e_1\in V$, called the unit, and a symmetric nondegenerate bilinear form $\eta\in (V^*)^{\otimes 2}$, called the metric, such that, chosen any basis $e_1,\ldots,e_{\dim V}$ of V, the following axioms are satisfied:
\begin{itemize}
\item[(i)] the maps $c_{g,n}$ are equivariant with respect to the $S_n$-action permuting the $n$ copies of~$V$ in $V^{\otimes n}$ and the $n$ marked points in $\oM_{g,n}$, respectively.
\item[(ii)] $\pi^* c_{g,n}( \otimes_{i=1}^n e_{\alpha_i}) = c_{g,n+1}(\otimes_{i=1}^n  e_{\alpha_i}\otimes e_1)$ for $1 \leq\alpha_1,\ldots,\alpha_n\leq \dim V$, where $\pi\colon\oM_{g,n+1}\to\oM_{g,n}$ is the map that forgets the last marked point.\\
Moreover $c_{0,3}(e_{\alpha}\otimes e_\beta \otimes e_1) =\eta(e_\alpha\otimes e_\beta) =:\eta_{\alpha\beta}$ for $1\leq \alpha,\beta\leq \dim V$.
\item[(iii)] $\gl^* c_{g_1+g_2,n_1+n_2}( \otimes_{i=1}^n e_{\alpha_i}) = c_{g_1,n_1+1}(\otimes_{i\in I} e_{\alpha_i} \otimes e_\mu)\eta^{\mu \nu} c_{g_2,n_2+1}( \otimes_{j\in J} e_{\alpha_j}\otimes e_\nu)$ for $2g_1-1+n_1>0$, $2g_2-1+n_2>0$ and $1 \leq\alpha_1,\ldots,\alpha_n\leq \dim V$, where $I \sqcup J = \{1,\ldots,n\}$, $|I|=n_1$, $|J|=n_2$, and $\gl\colon\oM_{g_1,n_1+1}\times\oM_{g_2,n_2+1}\to \oM_{g_1+g_2,n_1+n_2}$ is the corresponding gluing map and where $\eta^{\alpha\beta}$ is defined by $\eta^{\alpha \mu}\eta_{\mu \beta} = \delta^\alpha_\beta$ for $1\leq \alpha,\beta\leq \dim V$.\\
\end{itemize}
\end{definition}

Remark that a notion of infinite rank partial CohFT, i.e. a partial CohFT with an infinite dimensional phase space $V$, requires some care. One needs to clarify what is meant by the matrix~$(\eta^{\alpha\beta})$ and to make sense of the, a priori infinite, sum over $\mu$ and $\nu$, both appearing in Axiom (iii). One possibility is demanding that the image of the linear map $V^{\otimes (n-1)}\to H^*(\oM_{g,n},\mbC) \otimes V^*$ induced by $c_{g,n}\colon V^{\otimes n}\to H^*(\oM_{g,n},\mbC)$ is contained in $H^*(\oM_{g,n},\mbC) \otimes \eta^\sharp(V)$, where $\eta^\sharp\colon V\to V^*$ is the injective map induced by the bilinear form $\eta$. Then in Axiom~(iii), instead of using an undefined bilinear form $(\eta^{\alpha\beta})$ on $V^*$, one can use the bilinear form on $\eta^\sharp(V)$ induced by $\eta$. This solves the problem with convergence.\\

A useful special case is the following. Let the basis $\{e_\alpha\}_{\alpha \in I}$ of $V$ be countable and, for any~$(g,n)$ in the stable range and each $e_{\alpha_1},\ldots,e_{\alpha_{n-1}} \in V$, let the set $\{\beta \in I\, |\, c_{g,n}(\otimes_{i=1}^{n-1} e_{\alpha_i}\otimes e_\beta)\neq 0\}$ be finite.  In particular this implies that the matrix $\eta_{\alpha\beta}$ is row- and column-finite (each row and each column have a finite number of nonzero entries), which is equivalent to $\eta^\sharp(V)\subseteq \mathrm{span}(\{e^\alpha\}_{\alpha \in I})$, where $\{e^\alpha\}_{\alpha \in I}$ is the dual ``basis''. Let us further demand that the injective map $\eta^\sharp\colon V \to \mathrm{span}(\{e^\alpha\}_{\alpha \in I})$ is surjective too, i.e. that a unique two-sided  row- and column-finite matrix $\eta^{\alpha\beta}$, inverse to $\eta_{\alpha\beta}$, exists (it represents the inverse map $(\eta^\sharp)^{-1}\colon\mathrm{span}(\{e^\alpha\}_{\alpha \in I})\to V$). Then the equation appearing in Axiom (iii) is well defined with the double sum only having a finite number of nonzero terms for each boundary divisor.\\

We will now construct an example of such infinite rank partial CohFT.
\begin{proposition}\label{prop:CohFT}
Let $\mu$ be a formal parameter. The classes $c_{g,n}(\otimes_{i=1}^n e_{a_i}) := \exp(\mu^2\Theta(a_1,\ldots,a_n))$ form an infinite rank partial cohomological field theory with a phase space $V=\mathrm{span}(\{e_a\}_{a \in \mbZ})$, where the unit is~$e_0$ and the metric, written in the basis $\{e_a\}_{a \in \mbZ}$, is $\eta_{a b}=\delta_{a+b,0}$.
\end{proposition}
\begin{proof}
Axioms (i) and (ii) follow directly from the definition of the classes $\Theta(a_1,\ldots,a_n)$. For Axiom (iii) notice that, for fixed $e_{a_1},\ldots,e_{a_{n-1}}$, we have $c_{g,n}(\otimes_{i=1}^{n-1} e_{a_i} \otimes e_b) = 0$ unless $b= -\sum_{i=1}^{n-1} a_i$. Moreover we have $\eta^{a b}= \delta_{a+b,0}$ and equation (\ref{eq:gluing-Theta}) implies Axiom (iii) where the double sum consists of just one term.
\end{proof}

In \cite{Bur15} it was shown how to associate an integrable system of evolutionary Hamiltonian PDEs, called the double ramification (DR) hierarchy, to any CohFT and in \cite{BDGR18} it was remarked how a partial CohFT is sufficient for the construction to work. Let us see how such construction generalizes to the infinite rank partial CohFT introduced in Proposition \ref{prop:CohFT}. Recall from \cite{Bur15,BR16a} that the DR Hamiltonian densities are the generating series
\begin{align}\label{eq:DR Hamiltonians}
g_{a,d}:=\sum_{\substack{g\ge 0,\,n\ge 1\\2g-1+n>0}}\frac{(-\eps^2)^g}{n!}\sum_{\substack{b,b_1,\ldots,b_n\in\mbZ\\ a_1,\ldots,a_n\in\mbZ}}\left(\int_{\DR_g\left(b,b_1,\ldots,b_n\right)}\lambda_g\psi_1^d c_{g,n+1}\left(e_a\otimes \otimes_{j=1}^n e_{a_j}\right)\right)\prod_{j=1}^n p_{b_j}^{a_j} e^{-ibx},
\end{align}
for $a\in \mbZ$ and $d\in \mbZ_{\geq 0}$, seen as formal power series in the formal variables $\eps,\mu,e^{ix},p^a_b$, $a,b \in \mbZ$.\\

Thanks to the fact that the intersection numbers appearing in equation (\ref{eq:DR Hamiltonians}) vanish unless $\sum b_j=-b$ and that, by formula (\ref{eq:lambdag-Hain}), the class $\lambda_g \DR_g(b,b_1,\ldots,b_n)$ is a polynomial in $b_1,\ldots,b_n$ homogeneous of degree $2g$, the above generating functions can be expressed uniquely (see e.g.~\cite{BR16a}) as a degree $0$ differential polynomial, i.e. a formal power series in $\eps,\mu$ and the new formal variables $u^a_k$, $a\in \mbZ$, $k\in \mbZ_{\geq 0}$, of degree $0$ with respect to the grading $\deg u^a_k=k$, $\deg\eps = -1$. The relation between the new variables $u^*_*$ and the old ones $p^*_*,e^{ix}$ is given by the formula $u^a_k = \d_x^k \left(\sum_{b\in \mbZ}p^a_b e^{ibx}\right)$. The expression of the operator $\d_x$ in the new variables $u^*_*$ is given by
$$
\d_x=\sum_{\substack{a\in \mbZ\\ j\geq 0}} u^a_{j+1} \frac{\d}{\d u^a_j}.
$$
Specifically, thanks to the fact that the intersection numbers appearing in equation (\ref{eq:DR Hamiltonians}) vanish unless $\sum a_j=-a$, we obtain $g_{a,d}\in \mbC[u^{<0}_*][[u^{\geq 0}_*,\eps,\mu]]^{[0]}$, where we put the superscript $[0]$ to denote the space of differential polynomials of degree $0$.\\

The DR Hamiltonians are defined as the local functionals $\og_{a,d}:=\int g_{a,d} dx$, which denote the equivalence classes of $g_{a,d}$ in the quotient vector space $\left(\mbC[u^{<0}_*][[u^{\geq 0}_*,\eps,\mu]]/(\Im(\d_x)\oplus \mbC)\right)^{[0]}$. Notice that, with respect to the formal variables $p^*_*,e^{ix}$, the symbol $\int g_{a,d} dx$ represents the coefficient of $e^{i 0 x}$ in the formal power series $g_{a,d}$.\\

A result of \cite{Bur15} says that the DR Hamiltonians mutually commute,
\begin{gather}\label{eq:DRcommute}
\{\og_{a_1,d_1},\og_{a_2,d_2}\} = 0, \qquad a_1,a_2\in \mbZ,\quad d_1,d_2 \in \mbZ_{\geq 0},
\end{gather}
with respect to the Poisson brackets of two local functionals $\of_1 = \int f_1 dx, \of_2 = \int f_2 dx$, with $f_1,f_2 \in \mbC[u^{<0}_*][[u^{\geq 0}_*,\eps,\mu]]^{[0]}$, given by
\begin{gather}
\{\of_1,\of_2\} = \int \left(\sum_{a_1,a_2\in \mbZ}\frac{\delta \of_1}{\delta u^{a_1}} \delta_{a_1+a_2,0} \ \d_x\left(\frac{\delta \of_2}{\delta u^{a_2}}\right)\right) dx,
\end{gather}
where $\frac{\delta \of}{\delta u^a} = \sum_{k\geq 0}(-\d_x)^k \frac{\d f}{\d u^a_k}$ for $\of = \int f dx$ and $f\in  \mbC[u^{<0}_*][[u^{\geq 0}_*,\eps,\mu]]^{[0]}$.\\

Now we make the following observation, specific for the infinite rank partial CohFT we are dealing with. For fixed $d\in\mbZ_{\geq0}$, let us collect the DR Hamiltonian densities $g_{a,d}$, for all $a\in\mbZ$, into a single generating function $g_d:=\sum_{a\in \mbZ} g_{a,d} e^{-iay}$ by use of the extra formal variable $e^{iy}$. Because the classes $c_{g,n+1}\left(e_a\otimes \otimes_{j=1}^n e_{a_j}\right)$ are polynomials in $a_1,\ldots,a_n$ of top degree $2g$, where in particular the coefficient of $\mu^{2j}$ is a polynomial of degree $2j$, and $\sum a_j=-a$, we can consider the formal change of variables
$$
u_{k_1,k_2} = \d_y^{k_2}\left( \sum_{a\in \mbZ} u^a_{k_1} e^{i a y} \right)= \d_x^{k_1} \d_y^{k_2} \left(\sum_{a,b\in \mbZ} p^a_b e^{iay+ibx}\right)
$$
and express $g_d$ uniquely as a differential polynomial in these new variables, specifically $g_d \in \mbC[[u_{*,*},\eps,\mu]]^{[(0,0)]}$, where $\deg u_{k_1,k_2} = (k_1,k_2)$, $\deg \eps = (-1,0)$, $\deg \mu = (0,-1)$. Naturally, we have
\begin{align*}
\d_x=&\sum_{k_1,k_2\geq 0} u_{k_1+1,k_2} \frac{\d}{\d u_{k_1,k_2}},\\
\d_y=&\sum_{k_1,k_2\geq 0} u_{k_1,k_2+1} \frac{\d}{\d u_{k_1,k_2}}.
\end{align*}
We will denote $u_{0,0}$ simply by $u$.\\

The DR Hamiltonian densities $g_{a,d}$ can be recovered from $g_d$ by the formula $g_{a,d}=\int \left(g_d e^{iay}\right)dy$, which extracts the coefficient of $e^{-iay}$ from $g_d$. Hence $\og_{a,d} = \iint\left(g_d e^{iay}\right)dxdy$. This suggest to restrict our attention to the Hamiltonians $\og_{0,d}$, whose densities depend on $e^{iy}$ through $u_{*,*}$ only. These are the simplest and most commonly considered kind of local functionals.\\

Let $\og_d = \og_{0,d} = \iint g_d\ dx dy$ be the equivalence class of $g_d$ in the quotient vector space $\left(\mbC[[u_{*,*},\eps,\mu]]/(\Im(\d_x)\oplus\Im(\d_y)\oplus\mbC)\right)^{[(0,0)]}$. Then, from equation (\ref{eq:DRcommute}), we deduce
\begin{gather}
\{\og_{d_1},\og_{d_2}\} = 0, \qquad d_1,d_2 \in \mbZ_{\geq 0},
\end{gather}
where the Poisson bracket of two local functionals $\of_1 = \iint f_1\ dxdy, \of_2 = \iint f_2\ dxdy$, with $f_1,f_2 \in \mbC[[u_{*,*},\eps,\mu]]^{[(0,0)]}$ is given by
\begin{gather}
\{\of_1,\of_2\} = \iint \left(\frac{\delta \of_1}{\delta u}  \d_x\left(\frac{\delta \of_2}{\delta u}\right)\right) dx dy,
\end{gather}
where $\frac{\delta \of}{\delta u} = \sum_{k_1,k_2\geq 0}(-\d_x)^{k_1}(-\d_y)^{k_2} \frac{\d f}{\d u_{k_1,k_2}}$ for $\of = \iint f\ dxdy$ and $f\in \mbC[[u_{*,*},\eps,\mu]]^{[(0,0)]}$.\\

The evolutionary PDEs generated via the above Poisson structure by the DR Hamiltonians~$\og_d$ are all compatible and have the form
\begin{gather}\label{eq:DRhierarchy}
\frac{\d u}{\d t_d} = \d_x \frac{\delta \og_d}{\delta u},\qquad d\in \mbZ_{\geq 0}.
\end{gather}


\section{The noncommutative KdV hierarchy and the main theorem}\label{section:ncKdV}

The classical construction of the KdV hierarchy as the system of Lax equations 
$$
\frac{\d L}{\d t_n}=\frac{\eps^{2n}}{(2n+1)!!}\left[\left(L^{n+1/2}\right)_+,L\right],\quad n\ge 1,
$$
where $L=\d_x^2+2\eps^{-2}u$, admits generalizations, called noncommutative KdV hierarchies, where one doesn't have the pairwise commutativity of the $x$-derivatives of the dependent variable. In what follows we will work with a specific example from the class of noncommutative KdV hierarchies.\\ 

The graded algebra of differential polynomials in two space dimensions introduced in Section~\ref{section:DR}, $\mbC[[u_{*,*},\eps,\mu]]$, where $\deg u_{k_1,k_2} = (k_1,k_2)$, $\deg \eps = (-1,0)$, $\deg \mu = (0,-1)$, can be endowed with the following graded associative Moyal star-product. Let $f,g \in \mbC[[u_{*,*},\eps,\mu]]$, with $\deg f = (i_1,i_2)$, $\deg g = (j_1,j_2)$, then define
\begin{gather}\label{eq:Moyal}
f* g:=f\ \exp\left(\frac{i \eps\mu}{2}(\overleftarrow{\d_x} \overrightarrow{\d_y}-\overleftarrow{\d_y} \overrightarrow{\d_x})\right)\ g =\sum_{n\ge 0}\sum_{k_1+k_2=n}\frac{(-1)^{k_2}(i\eps\mu)^n}{2^n k_1!k_2!}(\d_x^{k_1} \d_y^{k_2} f) (\d_x^{k_2}\d_y^{k_1} g),
\end{gather}
with $\deg{(f * g)} = (i_1+j_1,i_2+j_2)$.\\

Let us consider the algebra of pseudo-differential operators of the form
$$
\sum_{i\le n}a_i*\d_x^i,\quad n\in\mbZ,\quad a_i\in\mbC[[u_{*,*},\mu]][[\eps,\eps^{-1}].
$$
Consider the operator $L:=\d_x^2+2\eps^{-2}u$. The noncommutative KdV hierarchy with respect to the Moyal star-product~\eqref{eq:Moyal} is is defined by (see e.g.~\cite{Ham05,DM00})
\begin{gather}\label{eq:ncKdV hierarchy}
\frac{\d L}{\d t_n}=\frac{\eps^{2n}}{(2n+1)!!}\left[\left(L^{n+1/2}\right)_+,L\right]_*,\quad n\ge 1,
\end{gather}
where we put the subscript $*$ in the notation for the commutator in order to emphasize that it is taken with respect to the noncommutative product $*$. The first equation of the hierarchy is
$$
\frac{\d u}{\d t_1}=\frac{1}{2}\d_x(u*u)+\frac{\eps^2}{12}u_{xxx}.
$$

\begin{theorem}\label{theorem:main}
The flows $\frac{\d}{\d t_d}$, $d\ge 1$, of the DR hierarchy (\ref{eq:DRhierarchy}) are given by the noncommutative KdV hierarchy~\eqref{eq:ncKdV hierarchy}.
\end{theorem}

\begin{proof}

We prove the theorem in two steps.\\

{\it Step 1}. Let us prove that the flow $\frac{\d}{\d t_1}$ of the DR hierarchy (\ref{eq:DRhierarchy}) is given by
\begin{gather}\label{eq:first flow of DR hierarchy}
\frac{\d u}{\d t_1}=\frac{1}{2}\d_x(u*u)+\frac{\eps^2}{12}u_{xxx}.
\end{gather}
For this we have to compute the integrals
\begin{multline}\label{Theta-integrals}
\int_{\oM_{g,n+1}}\lambda_g\psi_1\Theta(0,a_1,\ldots,a_n)^k\DR_g(0,b_1,\ldots,b_n)=\\
=(2g-2+n)\int_{\oM_{g,n}}\lambda_g\Theta(a_1,\ldots,a_n)^k\DR_g(b_1,\ldots,b_n).
\end{multline}
Relation (\ref{eq:lambdag-Theta-relation}) implies that integral~\eqref{Theta-integrals} is nonzero only if $n=3$ and $k=g$ or if $n=2$, $k=0$ and $g=1$. In the second case integral~\eqref{Theta-integrals} is equal to $\frac{b_1^2}{12}$, which gives the second term on the right-hand side of equation~\eqref{eq:first flow of DR hierarchy}.\\

Regarding integral~\eqref{Theta-integrals} for $n=3$ and $k=g$, by Theorem~\ref{eq:quadratic DR integral}, we have
\begin{align*}
&\sum_{g\ge 0}\sum_{ a_1, a_2, b_1, b_2\in\mbZ}\frac{(-\eps^2\mu^2)^g}{g!}\int_{\oM_{g,4}}\lambda_g\psi_1\Theta(0,-a_1-a_2,a_1,a_2)^g\DR_g(0,-b_1-b_2,b_1,b_2)p^{ a_1}_{ b_1}p^{ a_2}_{ b_2}=\\
&\hspace{2cm}=\sum_{g\ge 0}\sum_{a_1,a_2,b_1,b_2\in\mbZ}(-\eps^2\mu^2)^g\frac{(a_2 b_1-a_1 b_2)^{2g}}{2^{2g}(2g)!}p^{ a_1}_{ b_1}p^{ a_2}_{b_2}=\\
&\hspace{2cm}=\sum_{g\ge 0}\sum_{k_1+k_2=2g}\sum_{ a_1, a_2, b_1, b_2\in\mbZ}\frac{(-1)^{k_2} (-\eps^2\mu^2)^g}{2^{2g} k_1!k_2!}( a_2 b_1)^{k_1}( a_1 b_2)^{k_2} p^{ a_1}_{ b_1}p^{ a_2}_{ b_2}=\\
&\hspace{2cm}=\sum_{g\ge 0}\sum_{k_1+k_2=2g}\frac{(-1)^{k_2}(-\eps^2\mu^2)^g}{2^{2g} k_1!k_2!}u_{k_1,k_2}u_{k_2,k_1}=\\
&\hspace{2cm}=u*u.
\end{align*}
Equation~\eqref{eq:first flow of DR hierarchy} is then proved.\\

{\it Step 2}. Let us now check that all other flows $\frac{\d}{\d t_d}$ of the DR hierarchy, for $d\geq 2$, are described by the noncommutative KdV hierarchy.\\ 

Let $f \in \mbC[[u_{*,*},\eps,\mu]]$ and let $\deg f =:(\deg_x f,\deg_y f)$ so that, in particular,
$$
\deg_x u_{a,b}=a,\quad\deg_y u_{a,b}=b,\quad \deg_x\eps=\deg_y\mu=-1.
$$
We see that both the flows of the DR hierarchy (\ref{eq:DRhierarchy}) and the flows of the noncommutative KdV hierarchy (\ref{eq:ncKdV hierarchy}) have the form
\begin{gather}\label{eq:general hierarchy}
\frac{\d u}{\d t_d}=P_d(u_{*,*},\mu,\eps)=\sum_{i\ge 0}P_{d,i}(u_{*,*},\mu)\eps^i,\quad d\ge 1,
\end{gather}
where $P_{d,i}(u_{*,*},\mu)$ are polynomials in the variables $u_{*,*}$ and $\mu$ satisfying 
\begin{align}
&P_1=\frac{1}{2}\d_x(u*u)+\frac{\eps^2}{12}u_{xxx},\label{eq:properties of hierarchies1}\\
&P_{d,0}=\d_x\left(\frac{u^{d+1}}{(d+1)!}\right),\label{eq:properties of hierarchies2}\\
&\deg_x P_{d,i}=i+1,\qquad\deg_y P_{d,i}=0\label{eq:properties of hierarchies3}.
\end{align}
It remains to check that a hierarchy of commuting flows of the form~\eqref{eq:general hierarchy}, satisfying properties~\eqref{eq:properties of hierarchies1}--\eqref{eq:properties of hierarchies3}, is unique. This is guaranteed by the following lemma.
\begin{lemma}
Suppose that $P(u_{*,*})$ is a polynomial in the variables $u_{*,*}$ of degrees $\deg_x P=d\ge 2$, $\deg_y P=q\ge 0$, and such that the flows 
\begin{align*}
\frac{\d u}{\d t}=&uu_x,\\
\frac{\d u}{\d\tau}=&P(u_{*,*}),
\end{align*}
commute. Then $P=0$.
\end{lemma}
\begin{proof}
Without loss of generality we can assume that $P$ is homogeneous with respect to an auxiliary gradation given by $\tdeg u_{a,b}:=1$. So we assume that $\tdeg P=k\ge 1$.\\ 

For a vector $\oa=(a_1,\ldots,a_k)\in\mbZ^k$ denote $|\oa|:=\sum a_i$. Let
\begin{align*}
\mcP_k:=&\{\lambda=(\lambda_1,\ldots,\lambda_k)\in\mbZ^k_{\ge 0}|\lambda_1\ge\ldots\ge\lambda_k\},\\
\mcP_{k,q}:=&\{\lambda\in\mcP_k||\lambda|=q\}.
\end{align*}
The set $\mcP_{k,q}$ is endowed with the lexicographical order. We will use the standard notation 
$$
m_p(\lambda):=\sharp\{1\le i\le k|\lambda_i=p\},\quad\lambda\in\mcP_k,\quad p\ge 0.
$$
The sequence $(m_0(\lambda),m_1(\lambda),\ldots)$ uniquely determines $\lambda$ that justifies the notation
$$
\lambda=(0^{m_0(\lambda)}1^{m_1(\lambda)}\ldots).
$$\\

Introduce a basis $f_{\lambda}$, $\lambda\in\mcP_k$, in the space of symmetric polynomials $\mbC[x_1,\ldots,x_k]^{S_k}$ by
$$
f_\lambda(x_1,\ldots,x_k):=\frac{1}{k!}\sum_{\sigma\in S_k}x_{\sigma(1)}^{\lambda_1}\ldots x_{\sigma(k)}^{\lambda_k},\quad\lambda\in\mcP_k.
$$
For $\lambda\in\mcP_k$ we call a polynomial $f\in\mbC[x_1,\ldots,x_k]$ $\lambda$-symmetric, if it is invariant with respect to the permutation of any pair of variables $x_i$ and $x_j$, $i\ne j$, such that $\lambda_i=\lambda_j$. We also introduce a notation for the symmetrization of a polynomial $f\in\mbC[x_1,\ldots,x_k]$ with respect to the variables $x_1,\ldots,x_a$, $1\le a\le k$:
$$
\Sym_{x_1,\ldots,x_a}f:=\frac{1}{a!}\sum_{\sigma\in S_a}f(x_{\sigma(1)},\ldots,x_{\sigma(a)},x_{a+1},\ldots,x_k).
$$\\ 

Making the substitutions $u(x,y)=\sum_{\alpha\in\mbZ}u^\alpha(y)e^{i\alpha x}$ and $t\mapsto\frac{t}{i}$, the flows $\frac{\d}{\d t}$ and $\frac{\d}{\d\tau}$ can be written in the form
\begin{align*}
\frac{\d u^\alpha}{\d t}=&\sum_{\substack{\alpha_1,\alpha_2\in\mbZ\\\alpha_1+\alpha_2=\alpha}}\alpha_1 u^{\alpha_1} u^{\alpha_2},&&\alpha\in\mbZ,\\
\frac{\d u^\alpha}{\d\tau}=&\sum_{\lambda\in\mcP_{k,q}}\sum_{\substack{\alpha_1,\ldots,\alpha_k\in\mbZ\\\sum\alpha_i=\alpha}}P_{\lambda}(\alpha_1,\ldots,\alpha_k)u^{\alpha_1}_{\lambda_1}\ldots u^{\alpha_k}_{\lambda_k},&&\alpha\in\mbZ,
\end{align*}
where $P_\lambda$ is a $\lambda$-symmetric polynomial in $\alpha_1,\ldots,\alpha_k$ of degree $d$. Here $u^\alpha_a:=\d_y^a u^\alpha$.\\

The commutator of the flows $\frac{\d}{\d t}$ and $\frac{\d}{\d\tau}$ is given by
\begin{align}
\frac{\d}{\d t}\frac{\d u^\alpha}{\d\tau}-\frac{\d}{\d\tau}\frac{\d u^\alpha}{\d t}=&\sum_{\beta\in\mbZ,\,b\ge 0}\d_y^b\left(\sum_{\substack{\beta_1,\beta_2\in\mbZ\\\beta_1+\beta_2=\beta}}\beta_1 u^{\beta_1} u^{\beta_2}\right)\frac{\d}{\d u^\beta_b}\sum_{\lambda\in\mcP_{k,q}}\sum_{\substack{\alpha_1,\ldots,\alpha_k\in\mbZ\\\sum\alpha_i=\alpha}}P_{\lambda}(\alpha_1,\ldots,\alpha_k)u^{\alpha_1}_{\lambda_1}\ldots u^{\alpha_k}_{\lambda_k}\notag\\
&-\sum_{\substack{\alpha_1,\alpha_2\in\mbZ\\\alpha_1+\alpha_2=\alpha}}\alpha_1\left(\sum_{\lambda\in\mcP_{k,q}}\sum_{\substack{\beta_1,\ldots,\beta_k\in\mbZ\\\sum\beta_i=\alpha_1}}P_{\lambda}(\beta_1,\ldots,\beta_k)u^{\beta_1}_{\lambda_1}\ldots u^{\beta_k}_{\lambda_k}\right)u^{\alpha_2}\notag\\
&-\sum_{\substack{\alpha_1,\alpha_2\in\mbZ\\\alpha_1+\alpha_2=\alpha}}\alpha_1 u^{\alpha_1}\sum_{\lambda\in\mcP_{k,q}}\sum_{\substack{\beta_1,\ldots,\beta_k\in\mbZ\\\sum\beta_i=\alpha_2}}P_{\lambda}(\beta_1,\ldots,\beta_k)u^{\beta_1}_{\lambda_1}\ldots u^{\beta_k}_{\lambda_k}=\notag\\
=&\sum_{\lambda\in\mcP_{k+1,q}}\sum_{\substack{\alpha_1,\ldots,\alpha_{k+1}\in\mbZ\\\sum\alpha_i=\alpha}}Q_{\lambda}(\alpha_1,\ldots,\alpha_{k+1})u^{\alpha_1}_{\lambda_1}\ldots u^{\alpha_{k+1}}_{\lambda_{k+1}},\label{sum}
\end{align}
where a polynomial $Q_\lambda$ is $\lambda$-symmetric.\\

Let $\tlambda:=\max\{\lambda\in\mcP_{k,q}|P_\lambda\ne 0\}$ and $\hlambda:=(\tlambda_1,\ldots,\tlambda_k,0)\in\mcP_{k+1,q}$. We see that the sum in line~\eqref{sum} has the form
$$
\sum_{\substack{\alpha_1,\ldots,\alpha_{k+1}\in\mbZ\\\sum\alpha_i=\alpha}}Q_{\hlambda}(\alpha_1,\ldots,\alpha_{k+1})u^{\alpha_1}_{\hlambda_1}\ldots u^{\alpha_{k+1}}_{\hlambda_{k+1}}+\sum_{\substack{\lambda\in\mcP_{k+1,q}\\\lambda<\hlambda}}\sum_{\substack{\alpha_1,\ldots,\alpha_{k+1}\in\mbZ\\\sum\alpha_i=\alpha}}Q_{\lambda}(\alpha_1,\ldots,\alpha_{k+1})u^{\alpha_1}_{\lambda_1}\ldots u^{\alpha_{k+1}}_{\lambda_{k+1}},
$$
and 
\begin{align*}
&Q_{\hlambda}(\alpha_1,\ldots,\alpha_{k+1})=\\
=&\Sym_{\alpha_{s+1},\ldots,\alpha_{k+1}}\Bigg[\sum_{i=1}^k(\alpha_iP_{\tlambda}(\alpha_1,\ldots,\alpha_i+\alpha_{k+1},\ldots,\alpha_k)-\alpha_iP_\tlambda(\alpha_1,\ldots,\alpha_k))\\
&\hspace{2.7cm}+\sum_{i=1}^s\alpha_{k+1}P_\tlambda(\alpha_1,\ldots,\alpha_i+\alpha_{k+1},\ldots,\alpha_k)-\alpha_{k+1}P_{\tlambda}(\alpha_1,\ldots,\alpha_k)\Bigg],
\end{align*}
where $s:=\sharp\{1\le i\le k|\tlambda_i\ge 1\}$.\\

We decompose the polynomial $P_\tlambda$ in the following way:
$$
P_\tlambda(\alpha_1,\ldots,\alpha_k)=\sum_{\substack{\oa\in\mbZ^s_{\ge 0},\,\mu\in\mcP_{k-s}\\|\oa|+|\mu|=d}}C_{\oa,\mu}\alpha_1^{a_1}\ldots\alpha_s^{a_s}f_\mu(\alpha_{s+1},\ldots,\alpha_k),\quad C_{\oa,\mu}\in\mbC.
$$
Let
\begin{align*}
P^{(p)}_\tlambda(\alpha_1,\ldots,\alpha_k):=&\sum_{\substack{\oa\in\mbZ^s_{\ge 0},\,|\oa|=p\\\mu\in\mcP_{k-s,d-p}}}C_{\oa,\mu}\alpha_1^{a_1}\ldots\alpha_s^{a_s}f_\mu(\alpha_{s+1},\ldots,\alpha_k),&& p\ge 0,&&\\
P^{(p,\mu)}_\tlambda(\alpha_1,\ldots,\alpha_k):=&\left(\sum_{\oa\in\mbZ^s_{\ge 0},\,|\oa|=p}C_{\oa,\mu}\alpha_1^{a_1}\ldots\alpha_s^{a_s}\right)f_\mu(\alpha_{s+1},\ldots,\alpha_k),&& p\ge 0,&& \mu\in\mcP_{k-s,d-p}.
\end{align*}
Let $\td:=\max\left\{p\ge 0\left|P_\tlambda^{(p)}\ne 0\right.\right\}$ and $\tmu=(0^{m_0}1^{m_1}\ldots):=\max\left\{\mu\in\mcP_{k-s,d-\td}\left|P_\tlambda^{(\td,\mu)}\ne 0\right.\right\}$. Note that
$$
\Sym_{\alpha_{s+1},\ldots,\alpha_{k+1}}(f_\mu(\alpha_{s+1},\ldots,\alpha_k)\alpha_{k+1})=f_\hmu(\alpha_{s+1},\ldots,\alpha_{k+1}),
$$
where $\hmu:=(0^{m_0}1^{m_1+1}2^{m_2}\ldots)$. It is now easy to see that
$$
Q_\hlambda^{\td,\hmu}(\alpha_1,\ldots,\alpha_{k+1})=\left(d+s-1+\sum_{i\ge 2}m_i\right)\left(\sum_{\oa\in\mbZ^s_{\ge 0},\,|\oa|=\td}C_{\oa,\tmu}\alpha_1^{a_1}\ldots\alpha_s^{a_s}\right)f_\hmu(\alpha_{s+1},\ldots,\alpha_{k+1}),
$$
which is nonzero, because $d\ge 2$ and $P^{(\td,\tmu)}_\tlambda\ne 0$. This contradicts the assumption that the flows~$\frac{\d}{\d t}$ and~$\frac{\d}{\d\tau}$ commute. The lemma is proved.
\end{proof}
This completes the proof of the theorem.
\end{proof}


\begin{thebibliography}{BDGR18}

\bibitem[Bur15]{Bur15} A. Buryak. {\it Double ramification cycles and integrable hierarchies}. Communications in Mathematical Physics~336 (2015), no.~3, 1085--1107.

\bibitem[BDGR18]{BDGR18} A. Buryak, B. Dubrovin, J. Gu\'er\'e, P. Rossi. {\it Tau-structure for the double ramification hierarchies}. Communications in Mathematical Physics~363 (2018), no.~1, 191--260.

\bibitem[BDGR20]{BDGR16} A. Buryak, B. Dubrovin, J. Gu\'er\'e, P. Rossi. {\it Integrable systems of double ramification type}. International Mathematics Research Notices, Volume 2020, Issue 24 (2020), pages 10381-10446, DOI:10.1093/imrn/rnz029.

\bibitem[BGR19]{BGR17} A. Buryak, J. Gu\'er\'e, P. Rossi. {\it DR/DZ equivalence conjecture and tautological relations}. Geometry~\& Topology 23 (2019), no.~7, 3537--3600.

\bibitem[BR16a]{BR16a} A. Buryak, P. Rossi. {\it Recursion relations for double ramification hierarchies}. Communications in Mathematical Physics 342 (2016), no. 2, 533--568.

\bibitem[BR16b]{BR16b} A. Buryak, P. Rossi. {\it Double ramification cycles and quantum integrable systems}. Letters in Mathematical Physics 106 (2016), no. 3, 289--317.

\bibitem[BR18]{BR18} A. Buryak, P. Rossi. {\it Extended $r$-spin theory in all genera and the discrete KdV hierarchy}. arXiv:1806.09825v2. 

\bibitem[BSSZ15]{BSSZ15} A. Buryak, S. Shadrin, L. Spitz, D. Zvonkine. {\it Integrals of {$\psi$}-classes over double ramification cycles}. American Journal of Mathematics~137 (2015), no.~3, 699--737.
 
\bibitem[DSZ20]{DSZ20} V. Delecroix, J. Schmitt, J. van Zelm. {\it admcycles -- a Sage package for calculations in the tautological ring of the moduli space of stable curves}. arXiv:2002.01709.

\bibitem[DM00]{DM00} A. Dimakis, F. M\"uller-Hoissen. {\it The Korteweg--de-Vries equation on a noncommutative space-time}. Physics Letters A 278 (2000), no. 3, 139--145.

\bibitem[FP00]{FP00} C. Faber, R. Pandharipande. {\it Logarithmic series and Hodge integrals in the tautological ring. With an appendix by Don Zagier}. Michigan Mathematical Journal~48 (2000), no.~1, 215--252.

\bibitem[Ful98]{Ful98} W. Fulton. {\it Intersection theory. Second edition}. Springer-Verlag, 1998.

\bibitem[Hai13]{Hai13} R. Hain. {\it Normal functions and the geometry of moduli spaces of curves}. Handbook of moduli. Vol.~I, 527--578, Adv. Lect. Math. (ALM), 24, Int. Press, Somerville, MA, 2013.

\bibitem[Ham05]{Ham05} M. Hamanaka. {\it Commuting flows and conservation laws for noncommutative Lax hierarchies}. Journal of Mathematical Physics~46 (2005), no.~5. 

\bibitem[HPS19]{HPS19} D. Holmes, A. Pixton, J. Schmitt. {\it Multiplicativity of the double ramification cycle}. Documenta Mathematica~24 (2019), 545--562.

\bibitem[JPPZ17]{JPPZ17} F. Janda, R. Pandharipande, A. Pixton, D. Zvonkine. {\it Double ramification cycles on the moduli spaces of curves}. Publications Mathématiques. Institut de Hautes \'Etudes Scientifiques~125 (2017), 221--266. 

\bibitem[KM94]{KM94} M. Kontsevich, Yu. Manin. {\it Gromov--Witten classes, quantum cohomology, and enumerative geometry}. Communications in Mathematical Physics~164 (1994), no. 3, 525--562.

\bibitem[LRZ15]{LRZ15} S.-Q. Liu, Y. Ruan, Y. Zhang. {\it BCFG Drinfeld-Sokolov hierarchies and FJRW-theory}. Inventiones Mathematicae 201 (2015), no. 2, 711--772.

\bibitem[Pix18]{Pix18} A. Pixton. {\it Generalized boundary strata classes}. Geometry of moduli, 285--293, Abel Symp., 14, Springer, Cham, 2018.

\end{thebibliography}
\end{document}